\documentclass[11pt]{article}

\usepackage[margin=1in]{geometry}
\usepackage{amsmath,amssymb,amsthm,amsfonts,mathrsfs,nccmath,bbm,bbold,mathbbol,dsfont}
\usepackage{graphicx,color,float,wrapfig}
\usepackage{enumerate,booktabs,multirow,siunitx,threeparttable}
\usepackage{url,fancyhdr,indentfirst}
\usepackage{natbib,comment,cite}
\usepackage{caption,subcaption}
\usepackage{bm}
\usepackage{tikz,tikz-qtree,pgfplots}
\usepackage{xr}
\usepackage[title]{appendix}
\usepackage[onehalfspacing]{setspace}
\usepackage[colorlinks=true,citecolor=blue]{hyperref}

\pgfplotsset{compat=1.18}
\usetikzlibrary{arrows.meta,positioning,shadows.blur}

\tikzset{
  font=\small,
  box/.style={
    rectangle,draw=black,rounded corners,thick,
    minimum width=5.5cm,align=center,inner sep=4pt,blur shadow={shadow blur steps=4},fill=gray!5
  },
  arr/.style={-Latex,thick,shorten >=2pt,shorten <=2pt},
  darr/.style={arr,dashed},
  mybox/.style={
    rectangle,draw=black,line width=1.2pt,minimum width=2.3cm,minimum height=0.9cm,
    inner sep=4pt,align=center,blur shadow={shadow blur steps=4},fill=gray!2 
  }
}

\makeatletter
\pgfarrowsdeclare{HalfBarbUp}{HalfBarbUp}
{
  \pgfarrowsleftextend{+-\pgflinewidth}
  \pgfarrowsrightextend{2.5\pgflinewidth}
}
{
  \pgfsetdash{}{0pt}
  \pgfpathmoveto{\pgfpoint{0pt}{0pt}}
  \pgfpathlineto{\pgfpoint{-4pt}{ 2.4pt}}
  \pgfpathlineto{\pgfpoint{-3pt}{ 0pt}}
  \pgfusepathqfill
}
\pgfarrowsdeclare{HalfBarbDown}{HalfBarbDown}
{
  \pgfarrowsleftextend{+-\pgflinewidth}
  \pgfarrowsrightextend{2.5\pgflinewidth}
}
{
  \pgfsetdash{}{0pt}
  \pgfpathmoveto{\pgfpoint{0pt}{0pt}}
  \pgfpathlineto{\pgfpoint{-4pt}{-2.4pt}}
  \pgfpathlineto{\pgfpoint{-3pt}{ 0pt}}
  \pgfusepathqfill
}
\makeatother

\renewcommand{\ge}{\geqslant}
\renewcommand{\le}{\leqslant}

\renewcommand{\epsilon}{\varepsilon}

\theoremstyle{plain}
\newtheorem{theorem}{Theorem}

\newtheorem{lemma}{Lemma}
\newtheorem{proposition}{Proposition}

\theoremstyle{definition}

\theoremstyle{remark}

\theoremstyle{definition}

\renewcommand{\cite}{\citet}

\renewcommand{\P}{\mathbb{P}}

\title{Asymptotic Probabilities of Attaining the Maximum in Heterogeneous Gaussian Samples}

\author{
  Chunxu Zhang\textsuperscript{1} \quad Baiqi Miao\textsuperscript{1} \quad and \quad Tiantian Mao\textsuperscript{1} \\
  \textsuperscript{1}Department of Statistics and Finance, School of Management, \\
  University of Science and Technology of China, Hefei, Anhui, China. \\
\textsf{zhangcx0627@mail.ustc.edu.cn},  
\textsf{bqmiao@ustc.edu.cn}, \textsf{tmao@ustc.edu.cn}
}

\date{\today}

\begin{document}
\maketitle 
\begin{abstract}
We study asymptotic probabilities of attaining the maximum in heterogeneous Gaussian samples. In the two-group setting, the first sample has variance $1$ and size $n_1$, while the second has variance $\sigma^2>1$ and size $n_2$. We investigate the probability that the maximum of the standard-variance group exceeds that of the high-variance group. Using the classical extreme-value normalization for Gaussian maxima together with a second-order comparison of the centering terms, we show that this probability admits a non-degenerate limit if and only if
$n_1\sim C n_2^{\sigma^2}(\log n_2)^{-(\sigma^2-1)/2}$ as $n_1,n_2\to\infty$
for some $C\in(0,\infty)$. In that regime, the limit admits an integral representation. Outside the critical regime, the comparison necessarily degenerates to $0$ or $1$. We then extend the analysis to finitely many independent Gaussian groups and obtain a generalized integral representation for the limiting winning probabilities. The results provide a complete asymptotic classification for this maximum-comparison problem.
\end{abstract}
 
\section{Introduction}

Let $X_1,\dots,X_n$ be independent Gaussian random variables whose variances may differ across groups, and consider the probability that a given group attains the overall maximum. This is a natural comparison problem for heterogeneous samples (often referred to as the \emph{winner problem}; see, e.g., \cite{davydov2024distribution}). When all variables are identically distributed, symmetry determines the answer immediately. Once the variances differ, however, the competition between tail heaviness and sample size becomes nontrivial.

Exact expressions are available for certain low-dimensional comparison problems: for $n=2$, the winning probability reduces to a univariate normal probability, while for $n=3$, it reduces to a bivariate normal orthant probability, which in the centered case admits the classical arcsine formula; see, for instance, \cite{nadarajah2008exact,habibi2011exact,nadarajah2019maximum}. 
To the best of our knowledge, however, a tractable closed-form evaluation of the winning probabilities is generally not available for \(n\ge 4\). Indeed, for independent Gaussian variables,
\[
\mathbb P\!\left(X_1=\max_{1\le j\le n}X_j\right)
=
\mathbb P\!\left(X_1-X_j\ge 0,\ \forall j=2,\ldots,n\right),
\]
which reduces the problem to the positive orthant probability of an \((n-1)\)-dimensional Gaussian vector. While such probabilities admit integral representations and numerical evaluation, they do not in general lead to a simple closed-form expression in higher dimensions. This makes asymptotic analysis a natural and, in effect, unavoidable approach.

At the same time, the asymptotic regime is of intrinsic interest. It reveals the large-sample balance between variance heterogeneity and sample size, identifies the critical scaling at which different groups remain asymptotically competitive, and yields explicit limiting winning probabilities. In this sense, the asymptotic theory developed here is not only a substitute for an unavailable exact formula, but also a structural description of the comparison problem itself.

The purpose of this paper is to determine, in explicit asymptotic form, when a high-variance group dominates the maximum and when several groups remain asymptotically competitive. Our starting point is the comparison problem itself, while classical extreme-value theory for Gaussian maxima serves as the main technical tool. After suitable centering and scaling, the maximum of an i.i.d.\ Gaussian sample converges in distribution to the Gumbel law; see, for example, \cite{leadbetter2012extremes,resnick1987,embrechts1997}. For the present problem, however, marginal extreme-value limits alone are not sufficient. To compare maxima coming from different Gaussian groups, one must analyze the relative position of their deterministic centering terms at a finer scale. In the two-group case studied here, this leads to a critical balance between the sample sizes and reveals an essential logarithmic correction beyond the leading polynomial order.

More precisely, we first consider two independent Gaussian groups. The first group consists of $n_1$ i.i.d.\ $N(0,1)$ variables, and the second consists of $n_2$ i.i.d.\ $N(0,\sigma^2)$ variables with $\sigma>1$. Writing $M^{(1)}_{n_1}$ and $M^{(2)}_{n_2}$ for the corresponding groupwise maxima, we study the asymptotic behavior of
\[
\mathbb P\!\left(M^{(1)}_{n_1}>M^{(2)}_{n_2}\right).
\]
If $n_1$ and $n_2$ are of the same order, then the larger-variance second group asymptotically wins with probability one. The nontrivial regime is therefore the critical one in which the lower-variance first group is allowed to grow faster. Our main two-group result shows that a non-degenerate limit exists if and only if
\[
n_1\sim C n_2^{\sigma^2}(\log n_2)^{-(\sigma^2-1)/2}
\]
for some constant $C>0$. In that case, the limiting winning probability admits the integral representation displayed in Theorem~\ref{prop:2}. This gives a complete classification of the asymptotic comparison: outside the critical regime, the limit necessarily degenerates to $0$ or $1$.

We then extend the analysis to finitely many independent Gaussian groups with different variances and sample sizes. Without loss of generality, taking one group as a baseline, we show that the vector of winning probabilities has a non-degenerate limit if and only if every other group is balanced against the baseline at its corresponding critical scale. Under this condition, the limiting winning probabilities admit a coupled integral representation. Thus, in both the two-group and the multi-group settings, the same principle governs the comparison problem: the leading polynomial growth of the sample sizes is not sufficient by itself, and the logarithmic correction is essential for a nontrivial limit.

The paper is organized as follows. Section~\ref{sec:2} analyzes the two-group problem. After recalling the Gaussian extreme-value normalization, we first identify the degenerate regime and then prove the complete characterization of the critical regime together with the integral representation of the limit. Section~\ref{sec:3} treats the multi-group case and derives the generalized limiting integral representations.

\section{Two-group case} \label{sec:2}
Let $X_1,\ldots,X_{n_1}$ be i.i.d.\ random variables with distribution $N(0,1)$, and let $X_{n_1+1},\ldots,X_{n_1+n_2}$ be i.i.d.\ random variables with distribution $N(0,\sigma^2)$, where $\sigma>1$. Assume that all these random variables are mutually independent, and write $n=n_1+n_2$. We are interested in the probability that the overall maximum is attained by the first observation from the first group, namely
\[
\mathbb P\!\left(\max_{1\le i\le n}X_i=X_1\right).
\]
Define
\[
M^{(1)}_{n_1}:=\max_{1\le i\le n_1}X_i,
\qquad
M^{(2)}_{n_2}:=\max_{n_1+1\le i\le n}X_i,
\]
and
\[
M_n:=\max_{1\le i\le n}X_i=\max\{M^{(1)}_{n_1},M^{(2)}_{n_2}\}.
\]
Since Gaussian distributions are continuous, ties occur with probability zero. By exchangeability within the first group,
\begin{equation} \label{eq:260221-1}
n_1\,\mathbb P\!\left( \max_{1\le i\le n} X_i = X_1 \right)
=
\mathbb P\!\left( M^{(1)}_{n_1} > M^{(2)}_{n_2} \right). \footnote{To see~\eqref{eq:260221-1}, note that the variables are continuous, and thus, the events
$\{\max_{1\le i\le n}X_i=X_j\}$,  $j=1,\ldots,n_1,$
are pairwise disjoint up to null sets. Their union is exactly the event
$ \{M^{(1)}_{n_1}>M^{(2)}_{n_2}\}
=
\{\max_{1\le i\le n}X_i=M^{(1)}_{n_1}\}.$
Hence,
\[
\mathbb P\!\left(M^{(1)}_{n_1}>M^{(2)}_{n_2}\right)
=
\sum_{j=1}^{n_1}\mathbb P\!\left(\max_{1\le i\le n}X_i=X_j\right).
\]
By exchangeability within the first group, all terms in the sum are equal to
$\mathbb P\!\left(\max_{1\le i\le n}X_i=X_1\right),$
which yields \eqref{eq:260221-1}.
}
\end{equation}
Thus the problem reduces to understanding the comparison probability on the right-hand side. 
Although an exact integral representation is available, it is not analytically transparent for the asymptotic comparison considered here.   Indeed, denoting by $\Phi$ the standard normal distribution function, we can write 
\[
\mathbb P\left(M^{(1)}_{n_1}>M^{(2)}_{n_2}\right)
=
\int_{\mathbb R}(\Phi(x/\sigma))^{n_2}\,\mathrm d\!\left(\Phi(x)^{n_1}\right),
\]
which is not analytically tractable. The asymptotic regime is therefore the natural one to study.

We shall use the classical extreme-value normalization for Gaussian maxima; see, for example, \cite{leadbetter2012extremes}. Let $F_\Lambda(x)=\exp(-e^{-x})$, $x\in\mathbb R$, denote the Gumbel distribution function. We call a random variable $\Lambda$ a \emph{Gumbel random variable} if $\Lambda$ follows the distribution $F_\Lambda$.
\begin{lemma}[Gaussian extreme-value theorem]\label{lem:EVT}
Let $Z_1,\ldots,Z_n$ be i.i.d.\ random variables with distribution $N(0,1)$. Then
\[
\frac{\max_{1\le i\le n}Z_i-b_n}{a_n}
\stackrel{\mathrm d}{\longrightarrow}\Lambda,
\]
 where $\Lambda\sim F_{\Lambda}$ is a Gumbel random variable, $\stackrel{\mathrm d}{\longrightarrow}$ denotes the convergence in distribution,
\[
a_n=(2\log n)^{-1/2},
\qquad{\rm and}\qquad
b_n=(2\log n)^{1/2}-\frac{\log\log n+\log(4\pi)}{2(2\log n)^{1/2}}.
\]
\end{lemma}

Applying Lemma~\ref{lem:EVT} to the two groups gives
\begin{equation}\label{eq:groupwise-EVT}
\frac{M^{(1)}_{n_1}-b_{n_1}}{a_{n_1}}
\stackrel{\mathrm d}{\longrightarrow}\Lambda,
\qquad
\frac{M^{(2)}_{n_2}-\sigma b_{n_2}}{\sigma a_{n_2}}
\stackrel{\mathrm d}{\longrightarrow}\Lambda,
\end{equation}
where the two limits are independent because the two groups are independent. Based on this fact, we immediately obtain  the easy case in which the two sample sizes are of the same order.

\begin{proposition}\label{prop:1}
Suppose that there exists $\gamma\in(0,1)$ such that $n_1\sim \gamma n$ and $n_2\sim (1-\gamma)n$ as $n\to\infty$. Then 
\[
\lim_{n\to\infty}\mathbb P\!\left(M^{(1)}_{n_1}>M^{(2)}_{n_2}\right) = 0.
\]
\end{proposition}

\begin{proof}
By \eqref{eq:groupwise-EVT},
and noting that $a_{n_1}\to0$ and $a_{n_2}\to0$, Slutsky's theorem gives
\[
M^{(1)}_{n_1}-b_{n_1}\xrightarrow{\mathbb P}0\qquad {\rm and}
\qquad
M^{(2)}_{n_2}-\sigma b_{n_2}\xrightarrow{\mathbb P}0\qquad {\rm as}~n\to\infty.
\]
It follows immediately that 
\[
\bigl(M^{(1)}_{n_1}-M^{(2)}_{n_2}\bigr)-\bigl(b_{n_1}-\sigma b_{n_2}\bigr)\xrightarrow{\mathbb P}0\qquad {\rm as}~n\to\infty.
\]
It remains to show that $b_{n_1}-\sigma b_{n_2}\to-\infty$. 
Using $b_m=\sqrt{2\log m}+o(1)$  as   $m\to\infty$, we obtain
\begin{align*}
b_{n_1}-\sigma b_{n_2}
&=\sqrt{2\log n_1}-\sigma\sqrt{2\log n_2}+o(1)\\
&=\sqrt{2\log n_1}\left(1-\sigma\sqrt{\frac{\log n_2}{\log n_1}}\right)+o(1)\to-\infty,\qquad {\rm as}~n\to\infty,
\end{align*}
where the limit is due to $\sigma>1 $ and ${\log n_2}/{\log n_1}\to 1$ as $n\to\infty$. Hence
$M^{(1)}_{n_1}-M^{(2)}_{n_2}\xrightarrow{\mathbb P}-\infty,$
and consequently, 
$\mathbb P\!\left(M^{(1)}_{n_1}>M^{(2)}_{n_2}\right)\to0.$
This completes the proof. 
\end{proof}

Proposition~\ref{prop:1}  suggests a degenerate regime: when the two sample sizes are of the same order, the probability that the standard first group attains the overall maximum tends to zero. 
Thus, a non-degenerate comparison can only arise under a more delicate asymptotic balance between the two sample sizes.

The next result identifies this balance completely. It gives a necessary and sufficient condition for the winning probability to converge to a non-degenerate limit, and therefore fully characterizes when the comparison is nontrivial and when it necessarily degenerates.
Specifically, let
\begin{equation}\label{eq:N1N2}
n_1\sim C n_2^{\sigma^2}(\log n_2)^{-(\sigma^2-1)/2}~~~{\rm as}~n=n_1+n_2\to\infty,
\end{equation}
for some constant $C>0$. 

\begin{theorem}\label{prop:2}
For $\sigma>1$,  the sequence
  $\mathbb P (M^{(1)}_{n_1}>M^{(2)}_{n_2} )$
converges to a limit in $(0,1)$ if and only if \eqref{eq:N1N2} holds for some $C\in(0,\infty)$.  Moreover, in this case,
\begin{equation}\label{eq:limit_p}
\lim_{n\to\infty}\mathbb P\!\left(M^{(1)}_{n_1}>M^{(2)}_{n_2}\right)
=
\mathbb P\!\left(\Lambda_1>\sigma^2(\Lambda_2-\kappa(C,\sigma))\right)
=
\int_0^\infty \exp\left(-y - e^{-\kappa(C,\sigma)} y^{1/\sigma^2}\right) \mathrm dy,
\end{equation}
where $\Lambda_1$ and $\Lambda_2$ are independent Gumbel random variables, and
\begin{align}
 \label{eq:0516-1}
 \kappa(C,\sigma)
=
\frac{1}{\sigma^2}\log\frac{C}{\sigma}
+\frac12\left(1-\frac{1}{\sigma^2}\right)\log(4\pi).
\end{align}
\end{theorem}
To prove Theorem~\ref{prop:2}, we need the following lemma. 
\begin{lemma}\label{lem2}
Under Condition \eqref{eq:N1N2}, it holds that 
\[
\frac{b_{n_1}-\sigma b_{n_2}}{\sigma a_{n_2}}
\to
\kappa(C,\sigma),
\]
where $\kappa(C,\sigma)$ is defined by \eqref{eq:0516-1}.
\end{lemma}

\begin{proof}
Write
$L_1:=\log n_1 $ and
$L_2:=\log n_2.$
By \eqref{eq:N1N2},
\begin{equation}\label{eq:L2-expansion}
L_1=\sigma^2L_2-\frac{\sigma^2-1}{2}\log L_2+\log C+o(1).
\end{equation}
Since $a_{n_2}=(2L_2)^{-1/2}$ and we can rewrite 
\[
b_m=\sqrt{2\log m}-\frac{\log\log m+\log(4\pi)}{2\sqrt{2\log m}},
\]
we have
\begin{align}\label{eq:v1-1}
\frac{b_{n_1}-\sigma b_{n_2}}{\sigma a_{n_2}}
&=T_{1,n}+T_{2,n}.
\end{align}
where
$T_{1,n}:=({\sqrt{2L_1}-\sigma\sqrt{2L_2}})/({\sigma a_{n_2}})$
and 
\begin{align*}
T_{2,n}:=-
\frac{1}{\sigma a_{n_2}}\left(
\frac{\log L_1+\log(4\pi)}{2\sqrt{2L_1}}
-
\sigma\frac{\log L_2+\log(4\pi)}{2\sqrt{2L_2}}
\right).
\end{align*}
Note that 
\begin{align*}
T_{1,n}
=
\frac{\sqrt{2L_2}}{\sigma}\bigl(\sqrt{2L_1}-\sigma\sqrt{2L_2}\bigr)
&=
\frac{2}{\sigma}\bigl(\sqrt{L_1L_2}-\sigma L_2\bigr)\\
&=
\frac{2L_2}{\sigma}\times\frac{L_1/L_2-\sigma^2}{\sqrt{L_1/L_2}+\sigma}
 =-\frac{\sigma^2-1}{2\sigma^2}\log L_2+\frac{1}{\sigma^2}\log C+o(1),
\end{align*}
where the last equality follows from  \eqref{eq:L2-expansion} and thus,
\[
\frac{L_1}{L_2}
=
\sigma^2-\frac{\sigma^2-1}{2}\frac{\log L_2}{L_2}+\frac{\log C}{L_2}+o(L_2^{-1}).
\]
Moreover, note that $L_1/L_2\to\sigma^2$, so $\log L_1=\log L_2+\log\sigma^2+o(1)$ and
 ${\sqrt{L_2}}/{\sqrt{L_1}}\to {1}/{\sigma}$.
Therefore,
\begin{align*}
T_{2,n}
&=-\frac{1}{2\sigma}\left[\sqrt{\frac{L_2}{L_1}}\bigl(\log L_1+\log(4\pi)\bigr)-\sigma\bigl(\log L_2+\log(4\pi)\bigr)\right]\\
&=
\frac{\sigma^2-1}{2\sigma^2}\log L_2
-\frac{1}{\sigma^2}\log\sigma
+\frac12\left(1-\frac{1}{\sigma^2}\right)\log(4\pi)
+o(1).
\end{align*}
Substituting  the expansions of $T_{1,n}$ and $T_{2,n}$ into \eqref{eq:v1-1}, the logarithmic terms in $\log L_2$ cancel, and we obtain
 $\frac{b_{n_1}-\sigma b_{n_2}}{\sigma a_{n_2}}
\to \kappa(C,\sigma)$. 
This completes the proof. 
\end{proof}
Now we are ready to prove Theorem~\ref{prop:2}.
\begin{proof}[Proof of Theorem~\ref{prop:2}]
We first prove the ``if" part. Assume that \eqref{eq:N1N2} holds. Define
\[
A_n:=\frac{M^{(1)}_{n_1}-b_{n_1}}{a_{n_1}},
\qquad
B_n:=\frac{M^{(2)}_{n_2}-b_{n_1}}{a_{n_1}}.
\]
Then it holds that 
\[
B_n
=
\frac{\sigma a_{n_2}}{a_{n_1}}\times\left(\frac{M^{(2)}_{n_2}-\sigma b_{n_2}}{\sigma a_{n_2}}
-
\frac{b_{n_1}-\sigma b_{n_2}}{\sigma a_{n_2}}\right)
\] and
\[
\mathbb P\!\left(M^{(1)}_{n_1}>M^{(2)}_{n_2}\right)=\mathbb P(A_n>B_n).
\]
Since
 ${\sigma a_{n_2}}/{a_{n_1}}=
 {\sigma}\sqrt{{\log n_1}/{\log n_2}}\to \sigma^{2},$
and Lemma~\ref{lem2} gives
 $  
( {b_{n_1}-\sigma b_{n_2}})/({\sigma a_{n_2}})\to\kappa(C,\sigma),
 $  
Slutsky's theorem yields
\[
A_n\stackrel{\mathrm d}{\longrightarrow}\Lambda
\qquad {\rm and}\qquad
B_n\stackrel{\mathrm d}{\longrightarrow}\sigma^2\left(\Lambda-\kappa(C,\sigma)\right),
\]
where $\Lambda$ is a Gumbel random variable. For each \(n\), \(A_n\) and \(B_n\) are independent, since they depend on disjoint independent groups of samples. Hence
\[
\mathcal L(A_n,B_n)
=
\mathcal L(A_n)\otimes \mathcal L(B_n)
\Rightarrow
\mathcal L(\Lambda_1)\otimes \mathcal L(\sigma^2(\Lambda_2-\kappa(C,\sigma))).
\]
Here \(\mathcal L(X)\) denotes the law of \(X\), and \(\mu\otimes\nu\) denotes the product measure of two probability measures \(\mu\) and \(\nu\). That is, \((A_n,B_n)\) converges jointly to a vector whose components are independent and distributed as \(\Lambda\) and \(\sigma^2(\Lambda-\kappa(C,\sigma))\), respectively. Therefore,
\[
\mathbb P(A_n>B_n)
\to
\mathbb P\!\left(\Lambda_1>\sigma^2(\Lambda_2-\kappa(C,\sigma))\right),
\]
where $\Lambda_1$ and $\Lambda_2$ are two independent Gumbel random variables. 
By standard manipulation, one can verify that 
\begin{align*}
\mathbb P\!\left(\Lambda_1>\sigma^2(\Lambda_2-t)\right)
= \int_0^\infty \exp\left(-y - e^{-t} y^{1/\sigma^2}\right) \mathrm dy.
\end{align*}
This proves \eqref{eq:limit_p} and completes the proof of the ``if" part.

We next prove the ``only if" part. Denote by
$f(n_2):=n_2^{\sigma^2}(\log n_2)^{-(\sigma^2-1)/2}.$ We show the result by considering the following three cases. 
\begin{itemize}
\item [(i)] Suppose $n_1/f(n_2)\to\infty$. Fix any $C>0$ and define $n_1'(n_2)=\lfloor Cf(n_2)\rfloor$, where \(\lfloor x\rfloor\) denotes the integer part of \(x\), i.e., the largest integer not exceeding \(x\). For all large $n_2$, we have $n_1\ge n_1'(n_2)$. Since the maximum over the first group is monotone in the sample size,
\[
\mathbb P\!\left(M^{(1)}_{n_1}>M^{(2)}_{n_2}\right)
\ge
\mathbb P\!\left(M^{(1)}_{n_1'}>{M^{(2)}_{n_2}}\right).
\]
Taking limits and using the ``if" part gives
\[
\liminf_{n\to\infty}\mathbb P\!\left(M^{(1)}_{n_1}>M^{(2)}_{n_2}\right)
\ge \P\!\left(\Lambda_1>\sigma^2(\Lambda_2-\kappa(C,\sigma))\right).
\]
Since $\kappa(C,\sigma)\to\infty$ as $C\to\infty$, the right-hand side can be made arbitrarily close to one. Therefore,
\(
\lim_{n\to\infty}\mathbb P\!\left(M^{(1)}_{n_1}>M^{(2)}_{n_2}\right)=1 
\) yielding a contradiction. 
\item [(ii)] Suppose that $n_1/f(n_2)\to0$. Fix any $C>0$ and define $n_1'(n_2)=\lfloor Cf(n_2)\rfloor$. Then $n_1\le n_1'(n_2)$ for all large $n_2$, so the same monotonicity argument gives
$\mathbb P(M^{(1)}_{n_1}>M^{(2)}_{n_2})
\le
\mathbb P(M^{(1)}_{n_1'}>M^{(2)}_{n_2}).$
Letting $n\to\infty$ and then $C\downarrow0$, we obtain
\( 
\limsup_{n\to\infty}\mathbb P\!\left(M^{(1)}_{n_1}>M^{(2)}_{n_2}\right)
\le0,
\)
and thus, \(
\lim_{n\to\infty}\mathbb P\!\left(M^{(1)}_{n_1}>M^{(2)}_{n_2}\right)=0 
\) yielding a contradiction. 
\item [(iii)] If $n_1/f(n_2)$ does not converge in $[0,\infty]$, then there exist two subsequences along which it converges to two distinct limits $C_1$ and $C_2$ in $[0,\infty]$. By the already established ``if" part and the conventions
$\kappa(0,\sigma):=-\infty$,
 and
$\kappa(\infty,\sigma):=+\infty,$
we assert that the corresponding comparison probabilities converge to two different limits,
\begin{align} \label{eq:0515-1}
 \mathbb P\!\left(\Lambda_1>\sigma^2(\Lambda_2-\kappa(C_1,\sigma))\right)
\neq
\mathbb P\!\left(\Lambda_1>\sigma^2(\Lambda_2-\kappa(C_2,\sigma))\right).   
\end{align}
 To see \eqref{eq:0515-1}, define
 $  
    G(C):=
    \mathbb{P}\left(\Lambda_1>\sigma^2(\Lambda_2-\kappa(C,\sigma))\right)$,  $
     C\in(0,\infty).
$
Noting that \( \kappa(C,\sigma)\) is strictly increasing on
\(C\in[0,\infty]\) and the random variable
\(
    W:=\Lambda_1-\sigma^2\Lambda_2
\)
has a continuous and strictly positive density on \(\mathbb{R}\), we have if \(0<C_1<C_2<\infty\), then
\(
    -\sigma^2\kappa(C_1,\sigma)>-\sigma^2\kappa(C_2,\sigma),
\)
and therefore,
 \(G\) is strictly increasing on \([0,\infty]\). That is, \eqref{eq:0515-1} holds, which implies that the sequence $\mathbb P (M^{(1)}_{n_1}>M^{(2)}_{n_2} )$ cannot converge. This yields a contradiction. 
\end{itemize}
Combining the above three cases, we have that $n_1/f(n_2)$ must converge to a limit in $(0,\infty)$. This completes the proof.\end{proof}

From the proof, we can see that outside the critical regime, the comparison necessarily degenerates. More precisely, if
\[
 \frac{n_1}{n_2^{\sigma^2}(\log n_2)^{-(\sigma^2-1)/2}}\to 0,
\]
then
$ 
\mathbb P\!\left(M^{(1)}_{n_1}>M^{(2)}_{n_2}\right)\to 0,
$ 
whereas if
\[
\frac{n_1}{n_2^{\sigma^2}(\log n_2)^{-(\sigma^2-1)/2}}\to \infty,
\]
then
$  
\mathbb P\!\left(M^{(1)}_{n_1}>M^{(2)}_{n_2}\right)\to 1.
$  
Thus, Theorem~\ref{prop:2} completely classifies the asymptotic behavior of the winning probability. 
A non-degenerate limit exists if and only if the sample sizes are balanced at the critical scale 
$n_2^{\sigma^2}(\log n_2)^{-(\sigma^2-1)/2}$; otherwise the comparison degenerates to $0$ or $1$. 
In this sense, the result identifies precisely when the smaller-variance first group can still compete asymptotically with the larger-variance second group.
In the critical regime, the limiting winning probability admits an integral representation as in (\ref{eq:limit_p}). 
In particular, the asymptotic balance depends not only on the leading polynomial order, but also on the logarithmic correction term.

\section{Multi-group case}
\label{sec:3}

 We now extend the analysis to multiple groups. Let $K\ge2$ be fixed. For each $k\in\{1,\ldots,K\}$, let
$X^{(k)}_1,\ldots,X^{(k)}_{n_k}$
be i.i.d.\ random variables with distribution $N(0,\sigma_k^2)$, where $\sigma_k>0$, and assume that the collections from different groups are mutually independent. Define
\[
M^{(k)}_{n_k}:=\max_{1\le i\le n_k}X^{(k)}_i,
\qquad k=1,\ldots,K,
\]
and let
\[
M_n:=\max_{1\le j\le K}M^{(j)}_{n_j}.
\]
  We are interested in the winning probabilities
\[
p_{n,k}:=\mathbb P\!\left(M^{(k)}_{n_k}>M^{(j)}_{n_j},\ \forall j\neq k\right),
\qquad k=1,\ldots,K.
\]
Since all distributions are continuous, ties occur with probability zero.

\begin{lemma}\label{lem:exchange-multi}
For each $k\in\{1,\ldots,K\}$,
\[
p_{n,k}=n_k\,\mathbb P(M_n=X^{(k)}_1).
\]
\end{lemma}

\begin{proof}
Fix $k\in\{1,\ldots,K\}$. Since the variables in group $k$ are exchangeable and ties occur with probability zero,
\[
\left\{M^{(k)}_{n_k}=\max_{1\le j\le K}M^{(j)}_{n_j}\right\}
=
\bigcup_{i=1}^{n_k}
\left\{\max_{1\le j\le K}\max_{1\le \ell\le n_j}X^{(j)}_\ell=X^{(k)}_i\right\},
\]
and the union is disjoint up to null sets. Summing the probabilities of these events and using exchangeability inside group $k$ yields the claim.
\end{proof}

To formulate the asymptotic result, 
we compare all groups relative to group $1$. For clarity of presentation, we restrict attention to the case
\[
\sigma_1=1,
\qquad
\sigma_k>1,\quad k=2,\ldots,K.
\]
Write $\boldsymbol p_n=(p_{n,1},\ldots,p_{n,K})$ and define $\boldsymbol\beta_n=(\beta_{n,1},\ldots,\beta_{n,K})\in(0,\infty)^K$ by
\[
\beta_{n,1}:=1,\qquad
\beta_{n,k}:=\frac{n_1}{n_k^{\sigma_k^2}(\log n_k)^{-(\sigma_k^2-1)/2}},
\quad k=2,\ldots,K.
\]
Here \(\beta_{n,k}\) measures the size of group \(k\) relative to the critical scale determined by group \(1\).

Theorem~\ref{prop:2} suggests that, in the multi-group setting, a non-degenerate limit can only arise when each group is balanced against the baseline group at its own critical scale. 
The next theorem shows that this condition is also sufficient, and yields a complete characterization of the limiting winning probabilities.

\begin{theorem}[Multi-group non-degenerate limit]\label{prop:multi}
Assume that $n_k\to\infty$ for every $k\in\{1,\ldots,K\}$. Then \(\boldsymbol p_n\) converges to a non-degenerate limit \(\boldsymbol p=(p_1,\ldots,p_K)\in(0,1)^K\) if and only if
\[
\boldsymbol\beta_n\to\boldsymbol C=(C_1,\ldots,C_K)\in(0,\infty)^K.
\]
Under this condition, for each \(k=1,\ldots,K\),
\[
p_k
=
\mathbb P\!\left(\sigma_k^2(\Lambda_k-\kappa(C_k,\sigma_k)) > \sigma_j^2(\Lambda_j-\kappa(C_j,\sigma_j)),\ \forall j\neq k\right),
\]
where \(\Lambda_1,\ldots,\Lambda_K\) are independent Gumbel random variables, and  \(C_1=1\) and \(\kappa(C_1,\sigma_1)=0\). Equivalently,
\[
p_k
=
\int_0^\infty \frac{e^{-\kappa(C_k,\sigma_k)}}{\sigma_k^2} x^{1/\sigma_k^2 - 1} \exp\left( - \sum_{j=1}^K e^{-\kappa(C_j, \sigma_j)} x^{1/\sigma_j^2} \right) \mathrm{d} x,
\qquad k=1,\ldots,K.
\]
\end{theorem}

\begin{proof}
We first prove the ``if'' part. Assume that $\boldsymbol\beta_n\to\boldsymbol C=(C_1,\ldots,C_K)\in(0,\infty)^K$, where $C_1=1$. For each $k\in\{1,\ldots,K\}$, define
\[
Y_{n,k}:=\frac{M^{(k)}_{n_k}-b_{n_1}}{a_{n_1}}.
\]
For $k=1$, \eqref{eq:groupwise-EVT} gives
\[
Y_{n,1}=\frac{M^{(1)}_{n_1}-b_{n_1}}{a_{n_1}}\stackrel{\mathrm d}{\longrightarrow}\Lambda_1.
\]
For $k\ge2$, write
\begin{align}
    \label{eq:0515-2}
    Y_{n,k}
=
\frac{\sigma_k a_{n_k}}{a_{n_1}}\cdot\frac{M^{(k)}_{n_k}-\sigma_k b_{n_k}}{\sigma_k a_{n_k}}
+
\frac{\sigma_k b_{n_k}-b_{n_1}}{a_{n_1}}.
\end{align}
The condition $\beta_{n,k}\to C_k$ is equivalent to
\begin{align*}
n_1\sim C_k n_k^{\sigma_k^2}(\log n_k)^{-(\sigma_k^2-1)/2}.
\end{align*}
By the same asymptotic expansion as in Lemma~\ref{lem2}, this yields
\begin{align}
    \label{eq:0515-3}
\frac{b_{n_1}-\sigma_k b_{n_k}}{\sigma_k a_{n_k}} \to \kappa(C_k, \sigma_k).
\end{align}
Moreover, the relation $\log n_1 \sim \sigma_k^2 \log n_k$ implies
\begin{align*}
\frac{\sigma_k a_{n_k}}{a_{n_1}} = \sigma_k \sqrt{\frac{\log n_1}{\log n_k}} \to \sigma_k^2,
\end{align*}
and thus,
\begin{align}
    \label{eq:0515-5}
\frac{\sigma_k b_{n_k}-b_{n_1}}{a_{n_1}} = - \frac{b_{n_1}-\sigma_k b_{n_k}}{\sigma_k a_{n_k}} \times \frac{\sigma_k a_{n_k}}{a_{n_1}} \to -\sigma_k^2 \kappa(C_k, \sigma_k).
\end{align}
Hence by Slutsky's theorem, substituting \eqref{eq:0515-3} and \eqref{eq:0515-5} into \eqref{eq:0515-2}  yields
\[
Y_{n,k}\stackrel{\mathrm d}{\longrightarrow}\sigma_k^2 \Lambda_k - \sigma_k^2 \kappa(C_k,\sigma_k),
\qquad k=1,\ldots,K.
\]
 Because the group maxima are based on mutually independent samples, Slutsky's theorem extends directly to the joint convergence of $(Y_{n,1}, \ldots, Y_{n,K})$. Let
\[
Z_k = \sigma_k^2(\Lambda_k - \kappa(C_k, \sigma_k)),
\qquad k=1,\ldots,K,
\]
where \(\kappa(C_1, \sigma_1) = 0\). Because $M^{(k)}_{n_k} > M^{(j)}_{n_j} \iff Y_{n,k} > Y_{n,j}$, we have
\[
p_{n,k}
=
\mathbb P\!\left(Y_{n,k}>Y_{n,j},\ \forall j\neq k\right)
\to
\mathbb P\!\left(Z_k > Z_j,\ \forall j\neq k\right).
\]

To explicitly compute this limiting probability, let $E_j=e^{-\Lambda_j}$, so that $E_1,\ldots,E_K$ are independent $\mathrm{Exp}(1)$ random variables. The event $Z_k > Z_j$ is equivalent to
\[
E_j > E_k^{\sigma_k^2/\sigma_j^2} \exp\left( \frac{\sigma_k^2 \kappa(C_k, \sigma_k)}{\sigma_j^2} - \kappa(C_j, \sigma_j) \right).
\]
Conditioning on $E_k = y$, we integrate the product of $\mathbb{P}(E_j > t \mid E_k=y) = \exp(-t)$ for all $j \neq k$:
\[
p_k = \int_0^\infty \exp(-y) \prod_{j \neq k} \exp\left( - \left( y e^{\kappa(C_k, \sigma_k)} \right)^{\sigma_k^2/\sigma_j^2} e^{-\kappa(C_j, \sigma_j)} \right) \mathrm{d} y.
\]
Applying the change of variables $x = (y e^{\kappa(C_k, \sigma_k)})^{\sigma_k^2}$ yields the symmetric generalized integral representation
\[
p_k = \int_0^\infty \frac{e^{-\kappa(C_k,\sigma_k)}}{\sigma_k^2} x^{1/\sigma_k^2 - 1} \exp\left( - \sum_{j=1}^K e^{-\kappa(C_j, \sigma_j)} x^{1/\sigma_j^2} \right) \mathrm{d} x.
\]
This proves the ``if'' part.

We now prove the ``only if'' part. Suppose $\boldsymbol p_n\to\boldsymbol p\in(0,1)^K$. Then no component can asymptotically vanish. In particular, for each $k\ge2$, neither $\beta_{n,k}\to0$ nor $\beta_{n,k}\to\infty$ is possible: by the two-group result in Theorem~\ref{prop:2}, \(\beta_{n,k}\to0\) implies \(p_1\to0\), while
\(\beta_{n,k}\to\infty\) implies \(p_k\to0\).
Hence each sequence $\{\beta_{n,k}\}$ is tight in $(0,\infty)$.
If $\boldsymbol\beta_n$ failed to converge, there would exist two subsequences converging to distinct limits $\boldsymbol C\neq\boldsymbol C'$ in $(0,\infty)^K$. By the ``if'' part, the corresponding subsequences of $\boldsymbol p_n$ would converge to two different limit vectors. This contradicts the assumed convergence of $\boldsymbol p_n$. Therefore $\boldsymbol\beta_n$ must converge in $(0,\infty)^K$.
\end{proof}

Theorem~\ref{prop:multi} shows that the multi-group problem admits a complete asymptotic classification, just as in the two-group case. 
A non-degenerate limit exists if and only if each group is balanced against the baseline group at the appropriate critical scale. 
Under this condition, the limiting winning probabilities take an explicit integral form generalizing the two-group case.

\section{Numerical Experiments}
\label{sec:numerical}

This section provides numerical validations of Theorem~\ref{prop:2}. We first present a controlled simulation study under independent Gaussian settings, followed by an empirical validation using real-world climatology data.

\subsection{Simulation Study}
To verify the theoretical predictions of Theorem~\ref{prop:2}, we simulate independent Gaussian maxima for two groups. The first group has unit variance, while the second group has standard deviation $\sigma \in \{1.2, 1.5, 2.0\}$. The sample size of the second group, $n_2$, varies over a wide logarithmic scale. To keep both groups asymptotically competitive, the first group's sample size is determined by the critical scaling constraint:
\[
n_1 = \lfloor C n_2^{\sigma^2} (\log n_2)^{-(\sigma^2-1)/2} \rfloor
\]
for multiple matching constants $C \in \{0.1, 1.0, 5.0\}$. For each parameter configuration $(n_2, \sigma, C)$, we generate $100,\!000$ independent trials. The maxima $M_{n_1}^{(1)}$ and $M_{n_2}^{(2)}$ are efficiently simulated via uniform quantile transformations.

\begin{figure}[htbp]
    \centering
    \caption{Simulated winning probabilities versus theoretical limits.}
    \includegraphics[width=\textwidth]{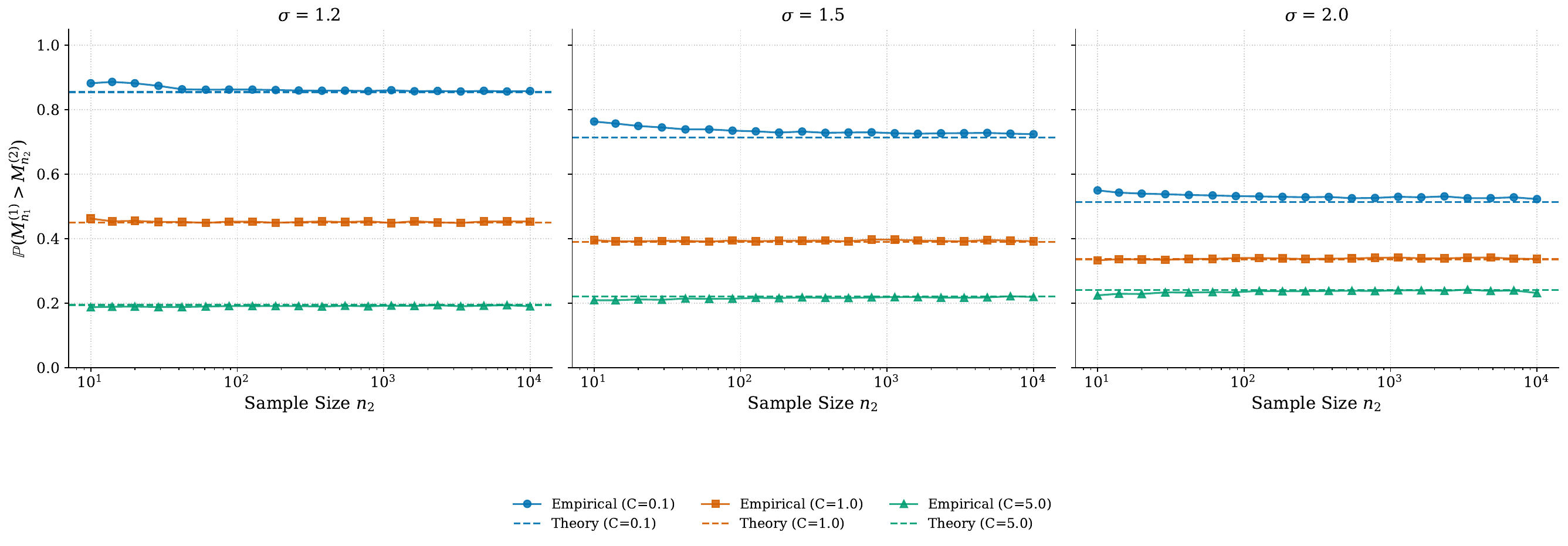}
    \label{fig:simulation}
    \vspace{0.5em}
    \parbox{\textwidth}{\footnotesize \textit{Notes.} The solid lines with markers display the empirical probability from $100,\!000$ simulated trials that Group 1 attains the maximum under varying scale multipliers $C$. The theoretical asymptotes predicted by Theorem~\ref{prop:2} are shown as dashed lines. The panels correspond to differing standard deviations $\sigma \in \{1.2, 1.5, 2.0\}$. Note the logarithmic horizontal axis scale.}
\end{figure}

As illustrated in Figure~\ref{fig:simulation}, the empirical probability that the first group attains the maximum converges precisely to the theoretical asymptotes established in Theorem~\ref{prop:2} as $n_2$ grows. The results confirm the accuracy of the critical scaling regime: as the variance disparity $\sigma$ increases, the sample size $n_1$ required for the lower-variance group to remain competitive grows enormously, fully prescribed by both the polynomial exponent and the $\sigma$-dependent logarithmic correction.

\subsection{Empirical Validation}
We further provide an empirical validation of Theorem~\ref{prop:2} using the NOAA Global Historical Climatology Network Monthly (GHCN-Monthly) v4 dataset. We focus on monthly average temperature records of United States stations from January 1980 to December 2025. To ensure spatial coherence, the pool is restricted to a geographical bounding box defined by latitudes $[30^\circ, 40^\circ)$ and longitudes $[-95^\circ, -75^\circ)$. For each station, we remove the seasonal cycle by subtracting the corresponding month-of-year means as well as any remaining linear trend across the observation period. We fit an AR(1) model to the detrended anomalies of each station and extract the resulting one-step innovations.

We split the pool of valid innovations into two distinct groups---a ``low variance" group (Group 1) and a ``high variance" group (Group 2). This partition is determined objectively via a 1D K-Means split on the innovation variances to minimize the within-cluster dispersion, allowing the standard deviation ratio $\sigma = \sigma_2/\sigma_1 > 1$ to form naturally from the dataset's right-skewed variance profile.

In the experiment, we vary $n_2 \in [5, 150]$ and set $n_1 = \lfloor C n_2^{\sigma^2} (\log n_2)^{-(\sigma^2-1)/2} \rfloor$ for $C \in \{0.1, 0.6, 3.0\}$. To estimate the winning probability for each pair of sample sizes $(n_1, n_2)$, we draw $10,\!000$ independent bootstrap samples with replacement from the respective groups and compute the frequency of $M_{n_1}^{(1)} > M_{n_2}^{(2)}$.

\begin{figure}[htbp]
    \centering
    \caption{Empirical winning probabilities versus theoretical limits.}
    \includegraphics[width=0.8\textwidth]{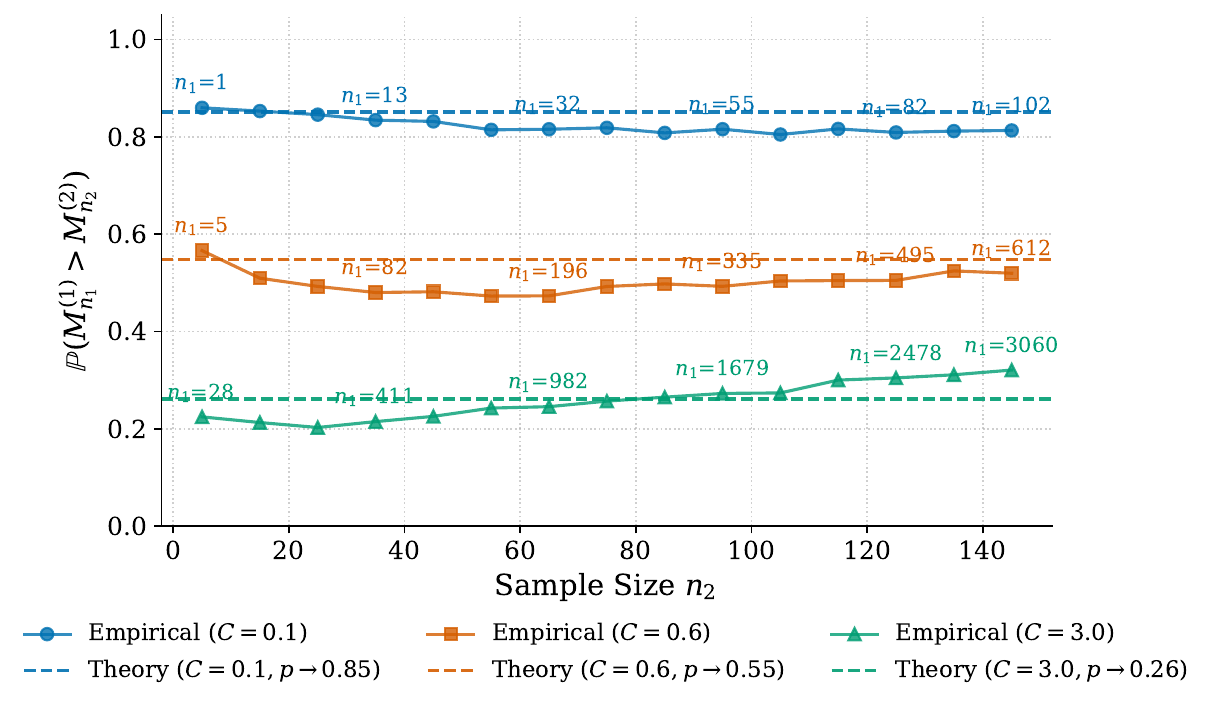}
    \label{fig:empirical_validation}
    \vspace{0.5em}
    \parbox{0.8\textwidth}{\footnotesize \textit{Notes.} The solid lines represent the bootstrap estimated probabilities over $10,\!000$ iterations from real climate station innovations under different matching constants $C$. Dashed lines denote the theoretical asymptotes from Theorem~\ref{prop:2}.}
\end{figure}

As illustrated in Figure~\ref{fig:empirical_validation}, the empirical results consistently stabilize toward the theoretical limit. Despite the inherent non-Gaussianity of real-world climate innovations, the rapid alignment between the bootstrap estimates and the predicted horizontal asymptotes confirms the robustness and predictive power of the extreme value limits derived in our framework.

\end{document}